\documentclass{aptpub} 

\authornames{V. LE}
\shorttitle{Coalescence times for the BGW process}

\numberwithin{equation}{section}
\begin{document}
    \title{Coalescence times for the \\ Bienaym\'e-Galton-Watson process}

    \author[Universit\'e de Provence]{By V.Le}
\address{
Laboratoire d'Analyse, Topologie, Probabilités (LATP/UMR 7353) \\
Universit\'e de Provence, 39, rue F. Joliot-Curie \\
F-13453 Marseille cedex 13, France\\
Email: levi121286@gmail.com}


\begin{abstract}
We investigate the distribution of the coalescence time (most recent common ancestor) for two individuals picked at random (uniformly) in the current generation of a continuous time Bienaym\'e-Galton-Watson process founded $t$ units of time ago. We also obtain limiting distributions as $t\rightarrow \infty $ in the subcritical case. We may also extend our results for two individuals to the joint distribution of coalescence times for any finite number of individuals sampled in the current generation.
\end{abstract}


\section{Introduction}
Random trees are mathematical objects that play an important role in many areas of
mathematics and other sciences. One of the most celebrated random trees is the Bienaym\'e-
Galton-Watson (BGW) tree, where the offspring of each vertex of the tree are independent
and indentically distributed (i.i.d) random integers. BGW tree plays a fundamental role in
both the theory and applications of stochastic processes. For more details, see e.g. $[1, 13]$.

One interesting and important approach to random trees is coalescence. In $[7]$, Lambert
has investigated the distribution of coalescence time for two individuals picked at random
(uniformly) in the current generation of a BGW process in the discrete setting. The
purpose of this note is to extend these results of Lambert to the case of continuous time
BGW process. The basic idea is the same as used in Lambert's paper, but we need some other techniques. We start a continuous time BGW process from a number $x$ of individuals at time $0$. Its law is denoted by $\mathbb P_x$ and $\mathbb P_x^{(t)}$ indicates that the current time is time $t$. If the current time contains at least two individuals, we pick uniformly within it two individuals, without replacement. We then compute the distribution of their coalescence time $T$ (if the current time contains less than two individuals, $T$ is set to $\infty $). In the subcritical case, the law $P^{qs}$ denoting the limit of the distributions $\mathbb P_x^{(t)}(\cdot \mid T<\infty )$ as $t\rightarrow \infty $ does not depend on $x$ and is called the quasi-stationary distribution. In section $3$, we specify the law of $T$ under $P^{qs}$. In section $4$, we extend our results to multivariate coalescence when $n$ individuals are sampled at the current time.

In this paper, the Lambert's results are not recalled. The reader should read again $[7]$
to compare the results in the discrete and continuous time cases. We also refer the reader
to several interesting closely related papers $[5, 9, 10, 14, 15]$.
\section{Distribution of the coalescence time}
Let $\mathbb N$ be the set of all natural numbers $\mathbb N = \{0, 1, 2, ...\}$. We consider a continuous time $\mathbb N$-valued branching process $Z=\{Z_t, t\geq 0\}$, where $t$ denotes time. Such a process is a Bienaym\'e-Galton-Watson process in which to each individual is attached a random vector describing its lifetime and its numbers of offspring. We assume that those random vectors are i.i.d.. The rate of reproduction is governed by a finite measure $\mu $ on $\mathbb N$, satisfying $\mu (1)=0$. More precisely, each individual lives for an exponential time with parameter $\mu (\mathbb N)$, and is replaced by a random number of children according to the probability $\mu (\mathbb N)^{-1} \mu $. Hence the dynamics of the continuous time Markov process $Z$ is entirely characterized by the measure $\mu $. For $x\in \mathbb N$, denote by $\mathbb P_x$ the law of $Z$ when $Z_0=x$. We have the following
proposition, which can be seen in $[1]$, chapter III (page $106$).
\begin{proposition}
The generating function of the process $Z$ is given by 
\begin{equation*}
\mathbb E_x(s^{Z_t}) = \psi _t(s)^x, \qquad s\in [0,1], x\in \mathbb N,
\end{equation*}
where 
\begin{equation*}
\frac{\partial \psi _t(s)}{\partial t} = \Phi (\psi _t(s)), \qquad \psi _0(s) =s,
\end{equation*}
and the function $\Phi $ is defined by
\begin{equation*}
\Phi (s) = \sum_{n=0}^\infty  (s^n-s)\mu (n), \quad s\in [0,1].
\end{equation*}
\end{proposition}
The continuous time BGW process $Z$ is called immortal if $\mu (0)=0$. In this paper, we always assume that $\mu (0)>0$. Let $\eta := \inf \{u>0: \Phi (u)=0\}$. Since $\Phi (0)= \mu (0)>0$, then we have $\eta >0$. Put
\begin{equation*}
F(t) := \int_0^t \frac{du}{\Phi (u)}, \qquad t<\eta .
\end{equation*}
Then the mapping $F: (0,\eta )\rightarrow (0,\infty )$ is bijective. We call $\varphi $ to be its inverse mapping. Moreover, $t\mapsto \psi _t(s)$ is the unique nonnegative solution of the integral equation
\begin{equation*} 
v(t)-\int_0^t \Phi (v(u))du =s, \qquad s\in [0,1], t\geq 0,
\end{equation*}
so that
\begin{equation*} 
\int_s^{\psi _t(s)}\frac{dv}{\Phi (v)}=t, \qquad s\in [0,1],s<\eta , t\geq 0.
\end{equation*}
Hence
\begin{equation*} 
\psi _t(s)=\varphi (t+F(s)), \qquad s\in [0,1],s<\eta , t\geq 0.
\end{equation*}
Note that the branching property implies that $\psi _{t_1+t_2}=\psi _{t_1}\circ \psi _{t_2}$.\\
Now, assume that the current generation is generation $t, t>0$. We consider two individuals $\sigma _1, \sigma _2$ at the present time, and ask when they coalesce, that is, how much time has elapsed since their common ancestor. In a more rigorous way, for $0<u\leq t$, denote by $\tau _u(\sigma _i)$ the (unique) parent of $\sigma _i$ at time $(t-u), i= 1,2$. The coalescence time $T(\sigma _1, \sigma _2)$ of $\sigma _1, \sigma _2$ is uniquely determined by 
$$T(\sigma _1, \sigma _2):= \inf\{u: 0<u\leq t, \tau _u(\sigma _1)=\tau _u(\sigma _2)\},$$
with the convention $\inf \emptyset =\infty $. We denote by $T$ the coalescence time of two individuals picked at random (uniformly) among the individuals which present in the current generation. If the current generation contains less than two individuals, $T$ is set to $\infty $. 

With the notation $\mathbb P^{(t)}$ indicates that $t$ is the current time, the distribution of $T$ is given in the following statement.
\begin{theorem}
For any $0<t_1\leq t_2\leq t, y\geq 1, y\in \mathbb N,$
\begin{equation*} 
\mathbb E^{(t)} (Z_t(Z_t-1)s^{Z_t-2}, T\leq t_1\mid Z_{t-t_2}=y) = y\psi _{t_2}^{'}(s) \psi _{t_2}(s)^{y-1} \frac{\psi _{t_1}^{''}(s)}{\psi _{t_1}^{'}(s)}, \qquad s\in [0,1).
\end{equation*}
The previous p.g.f can be inverted as follow, for any $p\geq 2$
\begin{align*}
\mathbb P^{(t)}(Z_t=p, T\in dt_1\mid Z_{t-t_2}=y)/dt_1 &=\\
y \sum_{n\geq 2} n\mu (n) \mathbb E \Big( \frac{Z_{t_2}^{(1)}(1)Z_{t_1}^{(2)}(n-1)}{p(p-1)}, Z_{t_2}^{(0)}(y-1)+Z_{t_2}^{(1)}(1)&+Z_{t_1}^{(2)}(n-1)=p \Big) ,
\end{align*}
where $Z^{(0)}, Z^{(1)}, Z^{(2)}$ are i.i.d branching processes distributed as $Z$, and the notation $Z_{t_2}^{(0)}(y-1)$ denotes the value taken by $Z^{(0)}$ at time $t_2$ when started at $y-1$.
\end{theorem}
\begin{remark}
When $t_2=t_1$, the above equation can be interpreted as follows. The amount $p$ of population at time $t$ is divided in three parts. An individual is marked at generation $t-t_1$ ($y$ possible choices), which is the candidate for the common ancestor of two random individuals of generation $t$ on $\{T\in dt_1\}$. The first part is the descendance at the current time of the $y-1$ remaining individuals. On $\{T\in dt_1\}$ the marked individual must be replaced immediately by $n$ offspring, $n\geq 2$. Then an individual is marked among the $n$ possible offspring of the previously marked ancestor. The descendance of this individual is the second part, and the descendance of the $n-1$ remaining others is the third part. On $\{T\in dt_1\}$, one of the two individuals sampled must be in the second part, and the other in the third part.
\end{remark}
\begin{proof}
To get the first equation, we use the same argument used in the proof of Theorem
$1$ in $[7]$. The second equation of the theorem is equivalent to
$$\mathbb E^{(t)} (Z_t(Z_t-1)s^{Z_t-2}, T\in  dt_1\mid Z_{t-t_2}=y)/dt_1 =$$
\begin{equation} 
y \sum_{n\geq 2} n\mu (n) \mathbb E \Big( Z_{t_2}^{(1)}(1)Z_{t_1}^{(2)}(n-1) s^{Z_{t_2}^{(0)}(y-1)+Z_{t_2}^{(1)}(1)+Z_{t_1}^{(2)}(n-1)-2} \Big) \quad \forall s\in (0,1). 
\end{equation}
Using the first result of the theorem, the left-hand side of $(2.1)$ equals
\begin{equation*} 
\mathbb E^{(t)} (Z_t(Z_t-1)s^{Z_t-2}, T\in  dt_1\mid Z_{t-t_2}=y)/dt_1 = y\psi _{t_2}^{'}(s) \psi _{t_2}(s)^{y-1} \frac{\partial }{\partial t_1} \Bigg( \frac{\psi _{t_1}^{''}(s)}{\psi _{t_1}^{'}(s)} \Bigg).
\end{equation*}
From the Proposition $1$ we have
\begin{align*}
&\frac{\partial \psi _{t_1}(s)}{\partial t_1} = \Phi (\psi _{t_1}(s)) \\
&\frac{\partial \psi _{t_1}^{'}(s)}{\partial t_1} = \Phi ^{'} (\psi _{t_1}(s)) \psi _{t_1}^{'}(s) \\
&\frac{\partial \psi _{t_1}^{''}(s)}{\partial t_1} = \Phi ^{''} (\psi _{t_1}(s)) \psi _{t_1}^{'}(s)^2+\Phi ^{'} (\psi _{t_1}(s)) \psi _{t_1}^{''}(s),
\end{align*}
so that
$$\frac{\partial }{\partial t_1} \Bigg( \frac{\psi _{t_1}^{''}(s)}{\psi _{t_1}^{'}(s)} \Bigg)= \frac{\psi _{t_1}^{'}(s)\frac{\partial \psi _{t_1}^{''}(s)}{\partial t_1}- \psi _{t_1}^{''}(s)\frac{\partial \psi _{t_1}^{'}(s)}{\partial t_1}}{\psi _{t_1}^{'}(s)^2}= \Phi ^{''} (\psi _{t_1}(s))\psi _{t_1}^{'}(s).$$
Then
\begin{align*}
\mathbb E^{(t)} (Z_t(Z_t-1)s^{Z_t-2}, &T\in  dt_1\mid Z_{t-t_2}=y)/dt_1 \\
&=y\psi _{t_2}^{'}(s) \psi _{t_2}(s)^{y-1} \Phi ^{''} (\psi _{t_1}(s))\psi _{t_1}^{'}(s)\\
&=y\psi _{t_2}^{'}(s) \psi _{t_2}(s)^{y-1} \psi _{t_1}^{'}(s) \sum_{n\geq 2} n(n-1) \mu (n)\psi _{t_1}(s)^{n-2}.
\end{align*}
Finally, the right-hand side of $(2.1)$ equals
\begin{align*}
&y \sum_{n\geq 2} n\mu (n) \mathbb E (s^{Z_{t_2}^{(0)}(y-1)}) \mathbb E ( Z_{t_2}^{(1)}(1) s^{Z_{t_2}^{(1)}(1)-1}) \mathbb E ( Z_{t_1}^{(2)}(n-1) s^{Z_{t_1}^{(2)}(n-1)-1} ) \\
&= y \sum_{n\geq 2} n\mu (n) \mathbb E_{y-1} (s^{Z_{t_2}}) \mathbb E_1 ( Z_{t_2} s^{Z_{t_2}-1}) \mathbb E_{n-1} ( Z_{t_1} s^{Z_{t_1}-1} ) \\
&= y \sum_{n\geq 2} n\mu (n) \psi _{t_2}(s)^{y-1} \psi _{t_2}^{'}(s) (n-1) \psi _{t_1}(s)^{n-2} \psi _{t_1}^{'}(s),
\end{align*}
which ends the proof.
\end{proof} 
\begin{corollary}
For any $0<t_1\leq t$,
$$\mathbb P_x^{(t)} (T\leq t_1)= x \int_0^1 ds(1-s)\frac{\psi _{t_1}^{''}(s)}{\psi _{t_1}^{'}(s)}\psi _{t}^{'}(s)\psi _{t}(s)^{x-1}.$$
In particular,\\
$\mathbb P_x^{(t)}$ (At least two extant individuals, a random pair has no common ancestor) =
$$x(x-1) \int_0^1 ds(1-s) \psi _{t}^{'}(s)^2\psi _{t}(s)^{x-2}.$$
\end{corollary}
\begin{proof}
See the proof of the corollary $1$ in $[7]$.
\end{proof}
\section{Quasi-stationary distribution}
In this section, we consider the limiting distribution of the coalescence time when the process is conditioned on $\{Z_t\geq 2\}$ and $t\rightarrow \infty $. Informally, this limit embodies the situation where the genealogy was founded a long time ago and is still not extinct, with at least two descendants at the present time. We will need some results on quasi-stationary distributions for the continuous time BGW process, which can be found in $[1, 4, 17]$. The reader may see more general results on quasi-stationary distributions, which have been obtained for continuous time Markov chains by $[16]$ and for semi-Markov processes by $[3]$. We also refer the reader to $[2, 8, 11]$ for the results on quasi-stationary distributions for population processes.

We consider the case $\psi _1^{'}(1)= \mathbb E_1(Z_1)<1$ (subcritical case) when $\mathbb E_1(Z_1 \log(Z_1))<\infty $. According to Theorem $6$ in $[17]$, there is a nonnegative sequence $(\alpha _k, k\geq 1)$ summing to $1$ such that
\begin{equation}
\lim_{t\rightarrow \infty } \mathbb P_x(Z_t=j\mid Z_t>0) = \alpha _j, \qquad \forall x\in \mathbb N, j\geq 1.
\end{equation}
The sequence $(\alpha _k, k\geq 1)$ is called the Yaglom limit of the process $Z$. If we define 
$$g(s)= \sum_{k\geq 1}\alpha _k s^k, \qquad s\in [0,1],$$
then $(3.1)$ deduces
\begin{equation*}
g(s)= \lim_{t\rightarrow \infty } \mathbb E_x (s^{Z_t} \mid Z_t>0) = \lim_{t\rightarrow \infty } \frac{\psi _t(s)-\psi _t(0)}{1- \psi _t(0)}, \qquad s\in [0,1].
\end{equation*}
We have the result:
\begin{proposition}
In the subcritical case when $\mathbb E_1(Z_1 \log(Z_1))<\infty $, we have for any $s \in  [0, 1]$,
\begin{equation}
\lim_{t\rightarrow \infty } \mathbb E_x (Z_t s^{Z_t-1} \mid Z_t>0) = g^{'}(s)\leq g^{'}(1)<\infty .
\end{equation}
\end{proposition}
The proof of Proposition $2$ can be found in $[1]$, chapter IV (page $170$). Under more
restrictive hypothesis that $\mathbb E_1(Z_1^2)<\infty $, we can give a very elementary and interesting
proof of $(3.2)$, which is provided by two following lemmas.
\begin{lemma}
For $t\geq 0$, let $\epsilon _t(s)$ be the function defined by
\begin{equation}
\frac{1-\psi _t(s)}{1- s}= \psi _t^{'}(1)- \epsilon _t(s), \qquad s\in [0,1).
\end{equation}
Then $\epsilon _t(s)$ is monotone decreasing, tend to zero when $s$ tend to one.
\end{lemma}
\begin{proof}
It follows from the fact that, for each $t$, $\psi _t(s)$ is increasing, convex, and $\psi _t(1)=1$.
\end{proof}
The equality $(3.3)$ is equivalent to
\begin{equation}
\frac{1-\psi _t(s)}{(1- s) \psi _t^{'}(1)}= 1- \frac{\epsilon _t(s)}{\psi _t^{'}(1)}.
\end{equation}
Replacing $s$ by $\psi _h(s)$ in $(3.4)$ we obtain
\begin{equation*}
\frac{1-\psi _t(\psi _h(s))}{(1- \psi _h(s)) \psi _t^{'}(1)}= 1- \frac{\epsilon _t(\psi _h(s))}{\psi _t^{'}(1)}\leq 1, \qquad t,h>0.
\end{equation*}
Note that $\psi _{t+h}(s) = \psi _t(\psi _h(s))$, and $\psi _{t+h}^{'}(1)= \psi _{t}^{'}(1)\psi _{h}^{'}(1)$, then
\begin{equation*}
\frac{1-\psi _{t+h}(s)}{(1-s) \psi _{t+h}^{'}(1)}= \frac{1-\psi _t(\psi _h(s))}{(1- \psi _h(s)) \psi _t^{'}(1)} \frac{1- \psi _h(s)}{(1-s) \psi _{h}^{'}(1)}\leq \frac{1- \psi _h(s)}{(1-s) \psi _{h}^{'}(1)}, \qquad t,h>0.
\end{equation*}
This implies that the sequence $(1- \psi _t(s))/((1-s) \psi _{t}^{'}(1))$ is monotone decreasing in $t$ and thus converges to a function $\chi (s)$. Letting $s=0$ we have 
\begin{equation*}
\chi (0)= \lim_{t\rightarrow \infty } \frac{\mathbb P_1(Z_t>0)}{\psi _t^{'}(1)}\geq 0.
\end{equation*}
\begin{lemma}
$\chi (0)$ is positive and for all $x\in \mathbb N$
 \begin{equation*}
\lim_{t\rightarrow \infty } \mathbb E_x(Z_t\mid Z_t>0) = g^{'}(1)= \frac 1{\chi (0)}.
\end{equation*}
\end{lemma}
\begin{proof}
We will follow the proof idea of Joffe as given in $[6]$. Note that
\begin{align*}
\chi (0)= \lim_{t\rightarrow \infty } \frac{1-\psi _t(0)}{\psi _t^{'}(1)}&=  \lim_{n\rightarrow \infty, n\in \mathbb N } \frac{1-\psi _n(0)}{\psi _n^{'}(1)}\\
&= \lim_{n\rightarrow \infty } \prod_{k=0}^{n-1} \big[ 1- \frac{\epsilon _1(\psi _k(0))}{\psi _1^{'}(1)}\big].
\end{align*}
Hence it follows that $\chi (0)>0$ if and only if the series $\sum_{k=0}^{\infty } \epsilon _1(\psi _k(0))$ converges. Since $\epsilon _t(s)\geq 0$ we get
\begin{equation*}
\frac{1-\psi _t(s)}{1- s}\leq  \psi _t^{'}(1), \qquad t\geq 0, s\in [0,1).
\end{equation*}
Letting $s=0$ we obtain
$$\psi _t(0)\geq 1- \psi _t^{'}(1), \qquad t\geq 0,$$
\begin{equation}
\epsilon _1(\psi _k(0))\leq \epsilon _1( 1- \psi _k^{'}(1)), \qquad k\geq 0.
\end{equation}
In the other hand, $\mathbb E_1(Z_1^2)<\infty $ implies that $\psi _1^{''}(1)<\infty $, then there exists a constant $C>0$ such that 
\begin{equation}
\epsilon _1(s)< C (1-s), \qquad s\in [0,1).
\end{equation}
From $(3.5)$ and $(3.6)$ we deduce that the series $\sum_{k=0}^{\infty } \epsilon _1(\psi _k(0))$ converges, so that $\chi (0)>0$. This implies that $\psi _t(0)\rightarrow 1$ as $t\rightarrow \infty $. Therefore
\begin{align*}
g(\psi _t(0)) &= \lim_{h\rightarrow \infty } \frac{\psi _{t+h}(0)- \psi _h(0)}{1-\psi _h(0)}\\
&= \lim_{h\rightarrow \infty } \frac{-(1-\psi _{t+h}(0))+(1- \psi _h(0))}{1-\psi _h(0)}\\
&= \frac{- \psi _{t+h}^{'}(1)+ \psi _{h}^{'}(1)}{\psi _{h}^{'}(1)}\\
&= -\psi _{t}^{'}(1)+1.
\end{align*}
Thus
\begin{equation*}
g^{'}(1)= \lim_{t\rightarrow \infty }\frac{g(\psi _t(0))-1}{\psi _t(0)-1}= \lim_{t\rightarrow \infty } \frac{-\psi _{t}^{'}(1)}{\psi _t(0)-1}= \frac 1{\chi (0)}.
\end{equation*}
\end{proof}
Denote by $\tilde{Z} $ the limiting value of $Z_t$ conditioned on $\{Z_t\geq 2\}$ as $t\rightarrow \infty $. We have
\begin{theorem}
In the subcritical case when $\mathbb E_1(Z_1 \log(Z_1))<\infty $, the quasi-stationary distribution $\mathbb P^{qs}$ of $T$ and $\tilde{Z} $ is defined by
\begin{equation*}
\mathbb P^{qs}(\tilde{Z}=p, T\in d h )= \lim_{t\rightarrow \infty } \mathbb P_x^{(t)} (Z_t=p, T\in d h\mid Z_t\geq 2), \qquad p\geq 2, h>0.
\end{equation*}
Then $\mathbb P^{qs}$ defines an probability distribution which does not depend on $x$ and satisfies
\begin{equation*}
\mathbb E^{qs} (\tilde{Z}(\tilde{Z}-1)s^{\tilde{Z}-2}, T\leq h)= \frac{g^{'}(s)}{1-g^{'}(0)} \frac{\psi _h^{''}(s)}{\psi _h^{'}(s)}.
\end{equation*}
In particular,
\begin{equation*}
\mathbb P^{qs} (T\leq h)= \frac 1{1-g^{'}(0)} \int_0^1 ds (1-s) \frac{\psi _h^{''}(s)}{\psi _h^{'}(s)} g^{'}(s).
\end{equation*} 
\end{theorem}
\begin{proof}
See the proof of Theorem $2$ in $[7]$.
\end{proof}
\section{Multivariate coalescence}
Assume that the current generation contains at least $n+1$ individuals, $n\geq 1$. We will present the distribution of coalescence times, when $n+1$ individuals are sampled uniformly and independently at the current time $t$. For $k = 1, 2,..., n$, we denote by $T_k$ the coalescence time of the first individual and the $(k +1)$-th individual, and by $T_k^*$ the $k$-th coalescence time. We have
\begin{theorem}
For any $0 < t_1 < t_2 < ... < t_n \leq  t$, the joint distribution of coalescence times $T_k$ is given by
\begin{align*}
&\mathbb E_x^{(t)}(Z_t(Z_t-1)...(Z_t-n) s^{Z_t-n-1}, T_1\in dt_1, ..., T_n\in dt_n)/ dt_1...dt_n= \\
& x\psi _t^{'}(s) \psi _t(s)^{x-1} \prod_{i=1}^n \psi _{t_i}^{'}(s) \big[ \sum_{k\geq 2} k(k-1)\mu (k) \psi _{t_i}(s)^{k-2}\big], \quad s\in [0,1).
\end{align*}
\end{theorem}
\begin{proof}
We will prove this theorem by induction since the formula holds when $n=1$ by Theorem $1$. We first condition on $\{Z_{t-t_n}= y\}$. We apply the second formula of Theorem $1$ to the last coalescence time $T_n$,
\begin{align*}
&\mathbb P^{(t)}(Z_t=p, T_1\in dt_1, ..., T_n\in dt_n\mid Z_{t-t_n}= y)/ dt_n= y \sum_{k\geq 2} k\mu (k)\times \\
&\mathbb E \Big( \frac{Z_{t_n}^{(1)}(1)Z_{t_n}^{(2)}(k-1)..(Z_{t_n}^{(2)}(k-1)-n+1)}{p(p-1)...(p-n)}, Z_{t_n}^{(0)}(y-1)+Z_{t_n}^{(1)}(1)+Z_{t_n}^{(2)}(k-1)=p, \\
& \qquad \qquad \qquad \qquad \qquad \qquad \qquad \qquad \qquad \qquad \qquad \qquad \qquad \qquad T_i\in dt_i, i\leq n-1 \Big),
\end{align*}
where the interpretation is as for $n=1$ (see Remark $1$): $y$ corresponds to the choice of the common ancestor of all individuals in generation $t-t_n$, $k$ is the number of offspring this ancestor had instantaneously at time $t-T_n$ and corresponds to the choice of the ancestor of the last individual within this offspring. The $n$ remaining individuals have to be found in the descendance of the $k-1$ remaining offspring. Then
\begin{align*}
&\mathbb E^{(t)}(Z_t(Z_t-1)...(Z_t-n) s^{Z_t-n-1}, T_1\in dt_1, ..., T_n\in dt_n\mid Z_{t-t_n}= y)/ dt_n= y \sum_{k\geq 2} k\mu (k)\times \\
&\mathbb E \Big( Z_{t_n}^{(1)}(1)Z_{t_n}^{(2)}(k-1)...(Z_{t_n}^{(2)}(k-1)-n+1) s^{Z_{t_n}^{(0)}(y-1)+Z_{t_n}^{(1)}(1)+Z_{t_n}^{(2)}(k-1)-n-1}, T_i\in dt_i, i\leq n-1 \Big) \\
&= y \sum_{k\geq 2} k\mu (k) \mathbb E\big(s^{Z_{t_n}^{(0)}(y-1)}\big) \mathbb E\big(Z_{t_n}^{(1)}(1) s^{Z_{t_n}^{(1)}(1)-1}\big)\times \\
& \mathbb E\big(Z_{t_n}^{(2)}(k-1)...(Z_{t_n}^{(2)}(k-1)-n+1)s^{Z_{t_n}^{(2)}(k-1)-n}, T_i\in dt_i, i\leq n-1\big)\\
&= y \psi _{t_n}(s)^{y-1} \psi _{t_n}^{'}(s)\sum_{k\geq 2} k\mu (k)\times \\
&\mathbb E\big(Z_{t_n}^{(2)}(k-1)...(Z_{t_n}^{(2)}(k-1)-n+1)s^{Z_{t_n}^{(2)}(k-1)-n}, T_i\in dt_i, i\leq n-1\big).
\end{align*}
By the induction hypothesis, the last expression equals
\begin{align*}
&y \psi _{t_n}(s)^{y-1} \psi _{t_n}^{'}(s)\sum_{k\geq 2} k\mu (k)\times  \\
&(k-1) \psi _{t_n}^{'}(s) \psi _{t_n}(s)^{k-2} \prod_{i=1}^{n-1} \psi _{t_i}^{'}(s) \big[ \sum_{j\geq 2} j(j-1)\mu (j) \psi _{t_i}(s)^{j-2}\big] dt_1...dt_{n-1}\\
&= y \psi _{t_n}(s)^{y-1} \psi _{t_n}^{'}(s) \prod_{i=1}^n \psi _{t_i}^{'}(s) \big[ \sum_{k\geq 2} k(k-1)\mu (k) \psi _{t_i}(s)^{k-2}\big] dt_1...dt_{n-1}.
\end{align*}
Hence the result follows by integrating w.r.t. to the distribution of $Z_{t-t_n}$ conditional on $\{Z_0=x\}$.
\end{proof}
\begin{theorem}
For any $0 < t_1 < t_2 < ... < t_n \leq  t$, the joint distribution of coalescence times $T_k^*$ is given by
\begin{align*}
&\mathbb E_x^{(t)}(Z_t(Z_t-1)...(Z_t-n) s^{Z_t-n-1}, T_1^*\in dt_1, ..., T_n^*\in dt_n)/ dt_1...dt_n= \\
& \frac{n! (n+1)!}{2^n} x\psi _t^{'}(s) \psi _t(s)^{x-1} \prod_{i=1}^n \psi _{t_i}^{'}(s) \big[ \sum_{k\geq 2} k(k-1)\mu (k) \psi _{t_i}(s)^{k-2}\big], \quad s\in [0,1).
\end{align*}
\end{theorem}
\begin{proof}
The proof is similar to that of Theorem $3$ above. We reason by induction since the
formula holds when $n = 1$ by Theorem $1$. We first condition on $\{Z_{t-t_n}=y\}$ and apply the
second formula of Theorem $1$ to the last coalescence time $T_n^*$,
\begin{align*}
&\mathbb P^{(t)}(Z_t=p, T_1^*\in dt_1, ..., T_n^*\in dt_n\mid Z_{t-t_n}= y)/ dt_n= \frac 12 y \sum_{k\geq 2} k\mu (k) \sum_{i=1}^n {n+1 \choose i} \sum_{1\leq j_1<..<j_{i-1}\leq n-1}\\
&\mathbb E \Big( \frac{Z_{t_n}^{(1)}(1)..(Z_{t_n}^{(1)}(1)-i+1)Z_{t_n}^{(2)}(k-1)..(Z_{t_n}^{(2)}(k-1)-n+i)}{p(p-1)...(p-n)}, Z_{t_n}^{(0)}(y-1)+Z_{t_n}^{(1)}(1)+Z_{t_n}^{(2)}(k-1) \\
& \quad =p, T_h^*(i)\in dt_h \text{ for } h\in \{j_1, .., j_{i-1}\} \text{ and } T_h^*(n+1-i) \in dt_h \text{ for } h\not\in \{j_1, .., j_{i-1}\}, h\leq n-1\Big),
\end{align*}
where the interpretation is as follows: $y$ corresponds to the choice of the common ancestor
of all individuals in generation $t- t_n$, $k$ is the number of offspring this ancestor had
instantaneously at time $t-T_n^*$ and corresponds to the choice of the ancestor of the last $i$
individuals within this offspring (there are ${n+1 \choose i}$ possible choices for the last $i$ individuals).
The $n + 1 - i$ remaining individuals have to be found in the descendance of the $k -1$ remaining offspring. For $m = 1, .., i - 1, T_{j_m}(i)$ is the $m$-th coalescence time of the last $i$ individuals, and for $h\not\in \{j_1, .., j_{i-1}\}, h\leq n-1$, $T_h^*(n+1-i)$ is a coalescence time of the $n + 1 -i$ remaining individuals. And we have to divide the expression by 2 because each
sample has been counted twice. We then have
\begin{align*}
&\mathbb E^{(t)}(Z_t(Z_t-1)...(Z_t-n) s^{Z_t-n-1}, T_1^*\in dt_1, ..., T_n^*\in dt_n\mid Z_{t-t_n}= y)/ dt_n\\
&=\frac 12 y \sum_{k\geq 2} k\mu (k) \sum_{i=1}^n {n+1 \choose i} \sum_{1\leq j_1<..<j_{i-1}\leq n-1}\\
&\mathbb E \Big( Z_{t_n}^{(1)}(1)..(Z_{t_n}^{(1)}(1)-i+1)Z_{t_n}^{(2)}(k-1)..(Z_{t_n}^{(2)}(k-1)-n+i) s^{Z_{t_n}^{(0)}(y-1)+Z_{t_n}^{(1)}(1)+Z_{t_n}^{(2)}(k-1)-n-1}, \\
& \quad T_h^*(i)\in dt_h \text{ for } h\in \{j_1, .., j_{i-1}\} \text{ and } T_h^*(n+1-i) \in dt_h \text{ for } h\not\in \{j_1, .., j_{i-1}\}, h\leq n-1\Big) \\
&=\frac 12 y \sum_{k\geq 2} k\mu (k) \sum_{i=1}^n {n+1 \choose i} \sum_{1\leq j_1<..<j_{i-1}\leq n-1} \mathbb E\big(s^{Z_{t_n}^{(0)}(y-1)}\big) \times \\
&\mathbb E\big(Z_{t_n}^{(1)}(1)..(Z_{t_n}^{(1)}(1)-i+1) s^{Z_{t_n}^{(1)}(1)-i}, T_h^*(i)\in dt_h \text{ for } h\in \{j_1, .., j_{i-1}\}\big)\times \\
& \mathbb E\big(Z_{t_n}^{(2)}(k-1)..(Z_{t_n}^{(2)}(k-1)-n+i)s^{Z_{t_n}^{(2)}(k-1)-n+i-1}, T_h^*(n+1-i) \in dt_h \\
&\qquad \qquad \qquad \qquad \qquad \qquad \qquad \qquad \qquad \qquad \qquad \qquad\text{ for } h\not\in \{j_1, .., j_{i-1}\}, h\leq n-1\big).
\end{align*}
By the induction hypothesis, the last expression equals
\begin{align*}
&\frac 12 y \sum_{k\geq 2} k\mu (k) \sum_{i=1}^n {n+1 \choose i} \sum_{1\leq j_1<..<j_{i-1}\leq n-1} \psi _{t_n}(s)^{y-1}\times \\
&\frac{(i-1)! i!}{2^{i-1}} \psi _{t_n}^{'}(s) \prod_{h\in \{j_1, .., j_{i-1}\}} \psi _{t_h}^{'}(s) \big[ \sum_{j\geq 2} j(j-1)\mu (j) \psi _{t_h}(s)^{j-2}\big]\times \\
& \frac{(n-i)! (n-i+1)!}{2^{n-i}} (k-1)\psi _{t_n}^{'}(s) \psi _{t_n}(s)^{k-2}\prod_{1\leq h\leq  n-1, h\not\in \{j_1, .., j_{i-1}\}} \psi _{t_h}^{'}(s) \big[ \sum_{j\geq 2} j(j-1)\mu (j) \psi _{t_h}(s)^{j-2}\big] \\
&\qquad \qquad \qquad \qquad \qquad \qquad \qquad \qquad \qquad \qquad \qquad \qquad \qquad \qquad \qquad \qquad dt_1 dt_2 .. dt_{n-1}\\
&=\frac 12 y \sum_{k\geq 2} k\mu (k) \sum_{i=1}^n {n+1 \choose i} \sum_{1\leq j_1<..<j_{i-1}\leq n-1} \frac{(i-1)! i!(n-i)! (n-i+1)!}{2^{n-1}} \psi _{t_n}^{'}(s)\psi _{t_n}(s)^{y-1}\times \\
&(k-1)\psi _{t_n}^{'}(s) \psi _{t_n}(s)^{k-2}\prod_{h=1}^{n-1} \psi _{t_h}^{'}(s) \big[ \sum_{j\geq 2} j(j-1)\mu (j) \psi _{t_h}(s)^{j-2}\big] dt_1 dt_2 .. dt_{n-1}\\
&=\frac 12 y \sum_{k\geq 2} k\mu (k) \sum_{i=1}^n {n+1 \choose i} {n-1 \choose i-1} \frac{(i-1)! i!(n-i)! (n-i+1)!}{2^{n-1}} \psi _{t_n}^{'}(s)\psi _{t_n}(s)^{y-1}\times \\
&(k-1)\psi _{t_n}^{'}(s) \psi _{t_n}(s)^{k-2}\prod_{h=1}^{n-1} \psi _{t_h}^{'}(s) \big[ \sum_{j\geq 2} j(j-1)\mu (j) \psi _{t_h}(s)^{j-2}\big] dt_1 dt_2 .. dt_{n-1}\\
&\frac{n! (n+1)!}{2^{n}} y \psi _{t_n}^{'}(s) \psi _{t_n}(s)^{y-1} \prod_{h=1}^{n} \psi _{t_h}^{'}(s) \big[ \sum_{j\geq 2} j(j-1)\mu (j) \psi _{t_h}(s)^{j-2}\big] dt_1 dt_2 .. dt_{n-1}.
\end{align*}
Hence the result follows by integrating w.r.t. to the distribution of $Z_{t-t_n}$ conditional on $\{Z_0 = x\}$.
\end{proof}
$\textbf{Acknowledgment}$. I am deeply grateful to my thesis advisor Professor Etienne Pardoux
for useful suggestions and constant encouragement during the preparation of this work. I would like also to thank the referee for a careful reading of the paper, as well as many valuable comments.
{}

 \end{document}